\documentclass{amsart}
\usepackage{amssymb}
\def\R{{\mathbb R}}

\def\N{{\mathbb N}}

\def\<{\langle}
\def\>{\rangle}
\def\P{\mathbb P}

\def\E{\mathbb E}

\def\0{\underline 0}

\def\X{\underline X}
\def\1{\underline 1}

\def\Y{\underline Y}

\def \X {{\widetilde X}}
\def \Y {{\widetilde Y}}
\def \U {{\widetilde U}}
\def \V {{\widetilde V}}

\newcommand{\bel}{\begin{equation}\label}
\newcommand{\nobel}{\begin{equation}}
\newcommand{\ee}{\end{equation}}
      \newtheorem{theorem}{Theorem}[section]
       
       \newtheorem{corollary}[theorem]{Corollary}
       
       \newtheorem{remark}{Remark}[section]
\theoremstyle{definition}

\title[Characterizations of the gamma and Kummer distributions]{Change of measure technique \\ in characterizations \\ of the Gamma and Kummer distributions}
\author{Agnieszka Piliszek,  Jacek Weso\l owski}
\address{Agnieszka Piliszek\\
Wydzia{\l} Matematyki i Nauk Informacyjnych\\
Politechnika Warszawska\\
Koszykowa 75\\
00-662  Warszawa, Poland}
\email{A.Piliszek@mini.pw.edu.pl}
\address{Jacek Weso{\l}owski\\
Wydzia{\l} Matematyki i Nauk Informacyjnych\\
Politechnika Warszawska\\
Koszykowa 75\\
00-662  Warszawa, Poland}
\email{wesolo@mini.pw.edu.pl}
\date{\today}

\begin{document}
\begin{abstract} If $X$ and $Y$ are independent random variables with distributions $\mu$ nd $\nu$ then $U=\psi(X,Y)$ and $V=\phi(X,Y)$ are also independent for some transformations $\psi$ and $\phi$. Properties of this type are known for many important probability distributions $\mu$ and $\nu$. Also related characterization questions have been widely investigated. These are questions of the form: Let $X$ and $Y$ be independent and let $U$ and $V$ be also independent. Are the distributions of $X$ and $Y$ necessarily $\mu$ and $\nu$, respectively?
	
	Recently two new properties and characterizations of this kind involving the Kummer distribution appeared in the literature.  For independent $X$ and $Y$ with gamma and Kummer distributions
	Koudou and Vallois in \cite{KV12} observed that $U=(1+(X+Y)^{-1})/(1+X^{-1})$ and $V=X+Y$ are also independent, and Hamza and Vallois in \cite{HV16} observed that $U=Y/(1+X)$ and $V=X(1+Y/(1+X))$ are independent. In \cite{KV11} and \cite{KV12} characterizations related to the first property were proved, while the characterizations in the second setting have been recently given in \cite{PW16}. These results were not fully satisfactory since in both cases technical assumptions on smoothness properties of densities of $X$ and $Y$ were needed.  In \cite{Wes15}, the assumption of independence of $U$ and $V$ in the first setting was weakened to constancy of regressions of $U$ and $U^{-1}$ given $V$ with no density assumptions. However, the additional assumption $\E\,X^{-1}<\infty$ was introduced.
	
	In the present paper we provide a complete answer to the characterization question in both settings without any additional technical assumptions regarding smoothness or existence of moments. The approach is, first, via characterizations exploiting some conditions imposed on regressions of $U$ given $V$, which are weaker than independence, but for which moment assumptions are necessary. Second, using a technique of change of measure we show that the moment assumptions can be avoided.
	
\end{abstract}

\maketitle

\section{Introduction}
In 2009 Koudou and Vallois (their paper, \cite{KV12}, was published in 2012) tried to describe a class of distributions of independent and positive random variables $X$ and $Y$ (their distributions) and functions $f$ such that $V=f(X+Y)$ and $U=f(X)-f(X+Y)$ are independent. Two special cases have already been known in the literature.
\begin{itemize}
\item Lukacs in \cite{L55}: $X$ and $Y$ are gamma distributed random variables and $f(x)=\log\,x$, $x>0$.
\item Matsumoto and Yor in \cite{MY01}: $X$ and $Y$ are generalized inverse Gaussian and gamma distributed random variables, respectively,  and $f(x)=1/x$, $x>0$.
\end{itemize}
Interestingly, due to the restriction $f>0$ imposed in \cite{KV12}, the first case was not identified there. Related characterizations were obtained  in \cite{L55} for the first property and  in \cite{LW00} for the second. Let us emphasize that neither smoothness of densities nor existence of moment assumptions were required in both these characterizations.

In \cite{KV12} the following new interesting case was discovered:
$X$ has the Kummer distribution $\mathcal{K}(a,b,c)$ with the density
$$
f_X(x)\propto \tfrac{x^{a-1}e^{-cx}}{(1+x)^{a+b}}\,I_{(0,\infty)}(x), \qquad a,b,c>0,
$$
and $Y$ has the gamma distribution $\mathcal{G}(b,c)$ with the density
$$
f_Y(y)\propto y^{b-1}e^{-cy}I_{(0,\infty)}(y)
$$
and $f(x)=\log\,\tfrac{x}{1+x}$.
Consequently, the random variables
\bel{juvi}
U=\tfrac{1+\tfrac{1}{X+Y}}{1+\tfrac{1}{X}}\qquad \mbox{and}\qquad V=X+Y
\ee
are independent. Moreover, $U$ has the beta first kind distribution $\mathrm{B}_I(a,b)$ with the density
$$
f_U(u)\propto u^{a-1}(1-u)^{b-1}I_{(0,1)}(u)
$$
and $V$ has the Kummer distribution, $\mathcal{K}(a+b,-b,c)$.

Note that the Kummer distribution $\mathcal{K}(\alpha,\beta,\gamma)$ is well-defined iff $\alpha,\gamma>0$ and $\beta\in\R$, Note that in the case $\beta>0$ it is just the natural exponential family parameterized by $\gamma$ and generated by the second kind beta distribution. It seems that the Kummer distribution appeared for the first time in 1953 in connection with Dyson's model of a disordered chain, \cite{D53}; its connection with random continued fractions was observed in \cite{MTW08}; it has been appearing in Bayesian models in queueing theory, \cite{AB97}, reliability analysis,  \cite{SBT95} or \cite{PSM96} and econometrics, \cite{AP83}; it has also been considered as a generalization of the gamma distribution, e.g. \cite{AK96}, \cite{G98} or \cite{NG07} (and applied to drought data in the latter paper); for $\alpha+\beta>0$ the distribution belongs to the generalized gamma convolutions and as such is infinitely divisible, see Th. 6.2.4 in the monograph \cite{B92}; in \cite{ASA07}, where it appears in connection with gamma finite mixtures, it is called Kobayashi's gamma-type distribution and thus related to problems of diffraction theory, \cite{K91}.

The independence property \eqref{juvi} was extended to matrix variate Kummer and Wishart distributions in \cite{Kou12}. For more information on the matrix variate Kummer distribution one can consult \cite{NC01} and  Ch. 3 of \cite{GN09} (the latter reference is for random matrices with complex entries). It is also known that in the univariate case, under appropriate smoothness assumptions on densities, a characterization counterpart of the property holds: if $X$ and $Y$ are independent positive random variables, and also $U$ and $V$, given by \eqref{juvi}, are  independent then $X\sim \mathcal{K}(a,b,c)$ and $Y\sim \mathcal{G}(b,c)$ for some positive constants $a,b,c$. Originally this result was proved in \cite{KV12} under requirements that the densities of $X$ and $Y$ are strictly positive and twice differentiable on $(0,\infty)$. Then, in \cite{KV11} it was proved under strict positivity of densities and local integrability of their logarithms. Letac in \cite{Le09} (see also Remark 2.2. in \cite{Wes15}) conjectured that such a characterization is possibly true with no assumptions on densities. In Section 3 we prove even a stronger result: instead assuming independence of $X$ and $Y$ and independence of $U$ and $V$ we impose weaker conditions which are expressed through certain relations between conditional moments of positive order of $U$ given $V$, while the assumption of independence $X$ and $Y$ is kept. Since, by its definition $U$ is bounded, no moment assumptions are necessary. This is a major difference when compared with \cite{Wes15}, where the constancy of regressions of $U$ and $U^{-1}$ given $V$ was considered. In this paper an additional requirement $\E\,X^{-1}<\infty$ was needed to assure that $\E(U^{-1}|V)$ is well defined. Our main characterization result related to \eqref{juvi} is given in Th. \ref{KV2}. Through a proper change of measure the original regression problem is reduced to the one solved in \cite{Wes15} with the required moment assumption being automatically satisfied. Finally, we utilize the results obtained under regression assumptions to obtain the  characterization by independence without any smoothness or integrability conditions.

Another independence property of the Kummer and gamma distributions has been recently discovered in \cite{HV16}: if $X\sim \mathcal{K}(a,b-a,c)$ and $Y\sim \mathcal{G}(b,c)$ then
\bel{def1}U=\tfrac{Y}{1+X}
\qquad \mbox{and}\qquad V=X\,(1+U) \ee
  are independent, $U\sim \mathcal{K}(b,a-b,c)$ and $V\sim\mathcal{G}(a,c)$. A related characterization assuming smoothness of densities (as well as its multivariate version with transformations defined in the language of directed trees) has been obtained in \cite{PW16} (some preliminary observations have been made also in \cite{Kou14}).  In Section \ref{HV} we provide regression version of the characterization, Th. \ref{tw1}, which is in parallel to the result from \cite{Wes15} we mentioned above. Th. \ref{tw2} is a counterpart of Th. \ref{KV1} recalled in the latter paragraph and uses a change of measure technique and Th. \ref{tw1}. However, in both these theorems we need to assume the existence of appropriate moments of $X$ and $Y$. The main result is Th. \ref{tw3}, in which we give characterization of Kummer and gamma distributions assuming only that $X$ and $Y$ are positive and independent and that $U$ and $V$ are independent. In the proof we again apply the technique of a change of measure. This time it allows to reduce the problem to the one considered in Th. \ref{tw2}, though no moment assumptions are required. We also note (Cor. \ref{tw_char}) that the latter result allows to improve the characterization known in the multivariate case.

\section{Characterizations related to the Hamza and Vallois property}\label{HV}
We consider a setting proposed in \cite{HV16} and developed in \cite{PW16}: $X$ and $Y$ are independent and $U$ and $V$ are defined in \eqref{def1}.
We begin with a characterization of Kummer and gamma laws by constancy of regression of $V$ and $V^{-1}$ given $U$. It is an analogue of Th. \ref{KV1} from \cite{Wes15}.
\begin{theorem}\label{tw1}
Let $X$ and $Y$ be independent positive non-degenerate random variables, such that ${\E X <\infty}$, $\E Y <\infty$ and $\E X^{-1} <\infty$. For $U$ and $V$ defined in \eqref{def1} suppose that there exist real constants $\alpha$ and $\beta$ such that
\bel{e1}\E(V|U) = \alpha \;\rm{ and }\ee
\bel{e2} \E(V^{-1}|U) = \beta. \ee

Then $\alpha\beta>1$ and there exists a constant $c>0$ such that
$$X\sim\mathcal{K}\left( \frac{\alpha\beta}{\alpha\beta -1},
c - \frac{\alpha\beta}{\alpha\beta -1},
\frac{\beta}{\alpha\beta -1}\right), \quad\mbox{and}\quad  Y\sim \mathcal{G}\left(c,\frac{\beta}{\alpha\beta - 1}\right).$$
\end{theorem}
\begin{proof}
Note that 
\bel{u+v}
U+V = X+Y.
\ee 
This observation and Eq. \eqref{e1} allows us to write that
\bel{pods}
 \E \left(X+1+Y|\frac{Y}{1+X}\right)= U + 1 + \E\left( V|U\right)=\frac{Y}{1+X} + 1 +\alpha.
\ee

Now, we will show that $\E Y^k<\infty$  for any $k\in\N$. We rewrite Eq. \eqref{pods} as
\bel{pods1}\E\left(X|\frac{Y}{1+X}\right)-\alpha = \frac{Y}{1+X}-\E\left(Y|\frac{Y}{1+X}\right).\ee
When we multiply both sides of \eqref{pods1} by $Y/(1+X)$ and take expectation, we get
$$\E Y^2\left(  \E (1+X)^{-2}-\E (1+X)^{-1} \right) = \E Y \left(\E X (1+X)^{-1}- \alpha \E (1+X)^{-1}\right) $$
and see that finitness of the second moment of $Y$ is guaranteed by the fact that $\E\,Y<\infty$, as assumed. If we multiply \eqref{pods1} again by the same factor, we conclude that $\E\,Y^2<\infty$ implies $\E\,Y^3<\infty$. So, we can obtain the finiteness of all the moments of $Y$ by multiplying \eqref{pods1} by $(Y/(1+X))^k$, $k=1,2,\ldots$ and taking expectations. (Of course, $\E (1+X)^{-k}<\infty$ since, under assumptions of the theorem, $1/(1+X)$ is bounded). Therefore, we can multiply Eq.\eqref{pods1} by $(Y/(1+X))^k$ and take expectation of both sides to obtain a recurrence relation
 \bel{wazne1}
 g_{k-1} - g_k = \alpha g_k + h_k g_{k+1} - g_k h_k\ee
 where $g_k = \E (1+X)^{-k}$ and $h_k = \frac{\E Y^{k+1}}{\E Y ^k}$, $k\geq 0$.

 By the definition of $V$, Eq. \eqref{e2} is equivalent to
 $$ \E\left(X^{-1}|\frac{Y}{1+X}\right) = \beta\left(1+\frac{Y}{1+X}\right).$$
 Again, since all required moments exist, we multiply the equation by $(Y/(1+X))^k$, $k\geq 0$ and take expectation to obtain
 \bel{e3} \E X^{-1}(1+X)^{-k}= \beta\E (1+X)^{-k} + \beta h_k \E (1+X)^{-(k+1)}.\ee
On the left-hand side of Eq. \eqref{e3} we use an elementary identity
\nobel \nonumber X^{-1} (1+X)^{-k} = X^{-1}-\sum_{j=1}^k (1+X)^{-j},\quad k\ge 1.\ee

Thus we have
\bel{jed} \E X^{-1} - \sum_{j=1}^k \E (1+X)^{-j} = \beta\E(1+X)^{-k} + \beta h_k \E (1+X)^{-(k+1)}\ee
and it holds for $k\geq 0$, with $\sum_{j=1}^0 =0$. Taking $k-1$ instead of $k$ in Eq. \eqref{jed} we get
\bel{dwaa} \E X^{-1} - \sum_{j=1}^{k-1} \E (1+X)^{-j} = \beta\E(1+X)^{-(k-1)} + \beta h_{k-1} \E (1+X)^{-k}.\ee
Now we subtract Eq. \eqref{jed} from Eq. \eqref{dwaa}. The factor $\E X^{-1}$ cancels out and we have
\bel{wazne2}
g_{k-1} - g_k = \frac{1}{\beta} g_k + h_k g_{k+1} - h_{k-1} g_k \quad\mathrm{for}\quad k\geq 1,
\ee
where $g_k$ and $h_k$ were defined earlier.
The left-hand sides of \eqref{wazne1} and \eqref{wazne2} are identical and so are the right-hand sides: $$\frac{1}{\beta} g_k + h_k g_{k+1} - h_{k-1} g_k =  \alpha g_k + h_k g_{k+1} - g_k h_k.$$
After simplification, and since $g_k>0$, we arrive at a linear relation of the form
$$ h_k = h_{k-1} + \alpha - \frac{1}{\beta}, \; k\geq 1.$$
Iterating the above equation one gets $h_k = \frac{k}{p} + \frac{c}{p}$, where  $p=\frac{\beta}{\alpha\beta-1}$,  $c$ is a real constant and $h_0=\E Y  =\frac{c}{p}$. It is easy to check, that $p$ is positive, since $\alpha\beta = \E V \E V^{-1} > \E V \left(\E V\right)^{-1}=1$ (here we use the fact that $X$ is not degenerate). Hence, $c>0$. Let us recall that $h_k = \frac{\E Y^{k+1}}{\E Y^k}$. We may conclude now, that $Y\sim\mathcal{G}(c,p)$.

In order to find the distribution of $X$, we insert into Eq. \eqref{wazne1} the values of $h_k$, and thus
 \bel{wazne3}
 pg_{k-1} + g_k\left(k+c -\frac{(1+\alpha)\beta}{\alpha\beta-1}\right) -g_{k+1} (k+c)= 0.\ee
 Comparing the last result with recurrence relation for function $${U(a,b,z) = \frac{1}{\Gamma (a)}\int_0^\infty e^{-zt} t ^ {a-1} (1+t)^{b-a-1}dt}$$
 in Abramowitz and Stegun (13.4.16)
 we can read at least one of the solutions of \eqref{wazne3}:
 $$\begin{array}{ccc}g_k &\propto& U\left( \alpha p, 1+\alpha p- k - c , p\right)\\
  &\propto& \int_0^\infty e^{-px} x^{\frac{1}{\alpha\beta-1}} (1+x)^{-(k+c)} dx.
\end{array}
$$
Consequently $X:=X_0$ where $X_0$ has a Kummer distribution: $\mathcal{K}\left( \alpha p, c- \alpha p,p\right)$, satisfies the assumptions of the theorem. The question is, if it is the only solution. To prove that, we define a function real $F$ by
$$F(z) := \sum_{k=1}^\infty z^k g_k =\E \left(\frac{z}{1+X-z}\right),$$
where the last equality follows from the definition of $g_k$. Note, that $F$ is well-defined at least for $z\in(-1,1)$. From the recurrence relation \eqref{wazne3} we obtain a differential equation for $F$:
\bel{rozn}
F'(z) z(1-z) = F(z) (pz^2 +dz +1 -c) + zcg_1  +pz^2,\;\;z\in(-1,1),
\ee
where $d = c-\alpha p$ and $g_1=\E (1+ X)^{-1}$.

As we have already observed, in the case $g_1=g_1^{(0)} := \E (1+X_0)^{-1}$  the function $F_0(z) = \E\left(\frac{1+X_0}{1+X_0-z}\right)$ solves Eq. \eqref{rozn}. Suppose that there exists a solution of Eq. \eqref{rozn} when $g_1\neq g_1^{(0)}$, say $F(z) = F_0(z) + H(z)$, $z\in(-1;1)$ and that this solution has the representation $F(z) = \E\left(\frac{1+X}{1+X-z}\right)$ for a positive random variable $X$. Then, the function $H$ satisfies
\bel{rozn2}H'(z)z(1-z)= H(z) (pz^2 + dz + 1-c)  +\tilde c z,\ee
where $\tilde{c} = cg_1 - cg_1^{(0)}$.
The general solution of \eqref{rozn2} is of the form
$$ H(z) = C(z) e^{-pz} |z|^{1-c} (1-z)^{1-\alpha p},$$
where $C$ is such that
\bel{nac}C'(z) = \tilde c e^{pz} |z|^{c-1}(1-z) ^{\alpha p -2}.\ee
We know that $H(1)= \E X^{-1} -\E X_0^{-1}$ and this value is well defined and finite. On the other hand
$\lim_{z\to 1} H(z)/C(z) = \infty$, as $1-\alpha p<0$. Thus, $\lim_{z\to 1} C(z) =0$.
Given Eq. \eqref{nac}, we have
\bel{calka}\int_0^1 C'(z) dz = -C(0).\ee
The integral on the left hand side is finite as $e^{pz}$ is bounded on $[0,1]$ and the rest of it is just a beta integral.

By the definition  of $H$ it follows that $H$ is analytic in its domain covering a neighbourhood, say $\mathcal{V}$, of $0$. Therefore $C(z) = |z|^{c-1} \sum_{n=0}^\infty a_n z^n$, $z\in \mathcal{V}$, where $a_n\in \R$, $n\ge 0$. Note, that
\bel{przed}\lim_{z\to 0^+} C(z) = \lim_{z\to 0 }\sum_{n=0}^\infty a_n z^{n-\delta}, \;\textrm{and} \; \delta = 1-c\in(-\infty,1).\ee
So
\bel{klamra}\lim_{z\to 0^+} C(z) =\left\{\begin{array}{ll}\displaystyle{\lim_{z\to 0} a_0} z^{-\delta}, & 0<c<1\\
0, & c\geq 1  \end{array}\right.= \left\{\begin{array}{ll}
+\infty, & a_0>0 \mathrm{\; and\;} 0<c<1,\\
-\infty, & a_0 <0\mathrm{\; and\;} 0<c<1 ,\\
0, & a_0=0 \mathrm{\; or\;} c\geq 1.
\end{array}\right.
\ee
For $c\neq 1$ due to \eqref{przed} the equalities in \eqref{klamra} are straightforward. For $c=1$  the definition of $H$ implies $H(z) = C(z) e^{-pz}(1-z)^{1-\alpha p}$ and thus, $C(0)=H(0)=F(0) -F_0(0) = 0$.

Any of the first two cases in \eqref{klamra} is impossible since the integral at the left hand side of \eqref{calka} is finite. The third equality implies that this integral in \eqref{calka} equals zero.  Since, according to \eqref{nac}, the integrand $C'(z)$ has constant sign on $(0,1)$, it follows that $\tilde{c}=0$. Thus $g_1=g_1^{(0)}$. Now, recurrence relation \eqref{wazne3} with $g_0=1$ and $g_1 = \E (1+X_0)^{-1}$ gives unique sequence $(g_k)_{k\geq 0}$ and   $g_k=U(\alpha p , 1+\alpha p - k -c, p)/U(\alpha p, 1+\alpha p - c, p)$, where $U$ is the function mentioned earlier. So we have $\E(1+X)^{-k} = \E(1+X_0)^{-k}$ for all $k\in \N$. Finally, since the support of the distribution of $1/(1+X)$ is a subset of $[0;1]$, we conclude that $X\stackrel{d}{=}X_0$.


\end{proof}

\begin{remark}
Under exactly the same  assumptions as in Th. \ref{tw1} above, the equation
$$
\beta h(s+1)(h(s-1)-\alpha h(s))-\beta h^2(s)+\E\,\tfrac{1}{X(1+X)^s}\,(h(s)-h(s+1))=0,\qquad s>1,
$$
where $h(t)=\E(1+X)^{-t}$, $t>0$, has been recently derived in  \cite{Kou14} as a first step in a search for the characterization.
\end{remark}

Two important consequences of Th. \ref{tw1} will be stated in the following theorems.

\begin{theorem}\label{tw2}
Let $X$ and $Y$ be independent positive non-degenerate random variables. For a fixed $r\in\R$ we assume that
\begin{enumerate}
	\item  ${\E X^{r+1} <\infty}$, ${\E Y^{r+1} <\infty}$ if $r\ge 1$;
	\item  ${\E X^{r+1} <\infty}$, ${\E Y^{r+1} <\infty}$  and  ${\E X^{r-1} <\infty}$, ${\E Y^{r-1} <\infty}$ if $r\in(-1,1)$;
	 	\item  ${\E X^{r-1} <\infty}$, ${\E Y^{r-1} <\infty}$ if $r\le -1$.
\end{enumerate} Define $U$ and $V$ through \eqref{def1}. If for $s\in\{r, r+1\}$
\bel{e1r}
\E(V^{s}|U) = \alpha_{s}\E (V^{s-1} | U)
\ee
for some real constants $\alpha_r$ and $\alpha_{r+1}$, then $\frac{\alpha_{r+1}}{\alpha_{r+1} - \alpha_r} >r$, $\alpha_{r+1}>\alpha_r$ and there exists a constant $c>0$ such that
$$X\sim \mathcal{K}\left( \frac{\alpha_{r+1}}{\alpha_{r+1} - \alpha_r} -r , c-\frac{\alpha_{r+1}}{\alpha_{r+1} - \alpha_r} +r,\frac{1}{\alpha_{r+1} - \alpha_r} \right),\;Y\sim \mathcal{G}\left(c,\frac{1}{\alpha_{r+1} - \alpha_r}\right).$$
\end{theorem}

\begin{proof}
We first notice that Eq. \eqref{e1r} for $s = r+1$ implies
$$\E \left( X^{r+1} (1+U) e^{tU}\right) = \alpha_{r+1} \E\left( X^r e^{tU}\right),\qquad t<0 ,$$
which can be rewritten in the integral form as
$$
\int_{(0,\infty)^2}\,x^{r+1}\left(1+\tfrac{y}{1+x}\right)e^{t\tfrac{y}{1+x}}\,\P_X(dx)\P_Y(dy)=\alpha_{r+1}\int_{(0,\infty)^2}\,x^re^{t\tfrac{y}{1+x}}\,\P_X(dx)\P_Y(dy).
$$

Denote  by $\X$ a random variable with distribution \bel{falka} d \P_{\X} (x) = \frac{x^r d\P_X (x)}{\E X^r}
\ee which is possibly defined on a different probability space than $X$ and $Y$. By $\Y$ we denote a random variable with the same distribution as $Y$ which is defined on the same probability space as $\X$ in such a way that $\X$ and $\Y$ are independent. Then we can rewrite last equality as
$$ \E (\V e^{t\U}) = \alpha_{r+1}\E(e^{t\U}),$$
where $\U = \Y/(1+\X)$ and $\V = \X(1+\U)$.
Given that, we have
\bel{do1}
\E(\V|\U) =\alpha_{r+1}.
\ee
Again, Eq. \eqref{e1r} with $s=r$ results in
$$\E(X^r e^{tU}) = \alpha_r \E (X^r X^{-1} (1+U)^{-1}e^{tU}). $$
With $\X$, $\Y$, $\U$ and $\V$ defined as before, the latter is equivalent to
\bel{do2}
\frac{1}{\alpha_r} = \E(\V^{-1}|\U).
\ee

By Eqs. \eqref{do1} and \eqref{do2}, the assumptions of Th. \ref{tw1} are satisfied for random variables $\X$ and $\Y$ with $\alpha = \alpha_{r+1}$, $\beta = \frac{1}{\alpha_r}$. Therefore,  $\X\sim \mathcal{K}\left( \frac{\alpha_{r+1}}{\alpha_{r+1} - \alpha_r} , c-\frac{\alpha_{r+1}}{\alpha_{r+1} - \alpha_r},\frac{1}{\alpha_{r+1} - \alpha_r} \right),\;Y\stackrel{d}{=}\Y\sim \mathcal{G}\left(c,\frac{1}{\alpha_{r+1} - \alpha_r}\right)$ and $\alpha\beta = \alpha_{r+1}/\alpha_r>1$. From \eqref{falka}, we conclude that random variable $X$ has density $f$ and
$$ f(x) \propto x^{\frac{\alpha_{r+1}}{\alpha_{r+1} - \alpha_r} -r}
(1+x)^{-c}e^{-\frac{1}{\alpha_{r+1} - \alpha_r}}
I_{(0,\infty)}(x)$$ The inequality $\frac{\alpha_{r+1}}{\alpha_{r+1} - \alpha_r} -r>0$ follows from integrability of $f$.
\end{proof}

\begin{remark}
	Similar extensions of the Lukacs and the Matsumoto-Yor properties were given in \cite{CH03} and \cite{CH04}, respectively. However these authors, instead of the change of measure method, which was applied above, used the standard but more cumbersome approach leading to differential equations for characteristic functions.
\end{remark}

\begin{corollary}
Let $X$ and $Y$ be independent positive non-degenerate random variables. For $U$ and $V$ defined in \eqref{def1} assume that there exist constants $a$ and $b$ such that
\bel{wn1}
\E(V|U)=a\quad\mbox{and}\quad\E(V^2|U)=b.
\ee
Then $a>0$, $b>a^2$ and there exists  $c>0$ such that
$$
X\sim \mathcal{K}\left(\tfrac{a^2}{b-a^2},c-\tfrac{a^2}{b-a^2},\,\tfrac{a}{b-a^2}\right)\quad\mbox{and}\quad Y\sim\mathcal{G}\left(c,\,\tfrac{a}{b-a^2}\right).
$$
\end{corollary}
\begin{proof}
Note that \eqref{wn1} implies \eqref{e1r} with $r=1$, $\alpha_1=a$ and $\alpha_2=\tfrac{b}{a}$.
\end{proof}
\begin{corollary}
Let $X$ and $Y$ be independent positive non-degenerate random variables and $\E\,X^{-2}<\infty$, $\E\,Y^{-2}<\infty$. For $U$ and $V$ defined in \eqref{def1} assume that there exist constants $a$ and $b$ such that
\bel{wn2}
\E(V^{-1}|U)=a\quad\mbox{and}\quad\E(V^{-2}|U)=b.
\ee
Then $\tfrac{b}{b-a^2}>-1$ and there exists a constant $c>0$ such that
$$
X\sim \mathcal{K}\left(\tfrac{b}{b-a^2}+1,\,c-\tfrac{b}{b-a^2}-1,\,\tfrac{ab}{b-a^2}\right)\quad\mbox{and}\quad Y\sim\mathcal{G}\left(c,\,\tfrac{ab}{b-a^2}\right).
$$
\end{corollary}
\begin{proof}
Note that \eqref{wn2} implies \eqref{e1r} with $r=-1$, $\alpha_0=\tfrac{1}{a}$ and $\alpha_{-1}=\tfrac{a}{b}$.
\end{proof}
This results can be compared with Th. 1.1 in \cite{LHH94} which is its analog for the Lukacs property.

\subsection{Characterization through independence}
In this subsection we answer the basic question: does the independence of $U$ and $V$ for independent, positive and non-degenerate $X$ and $Y$ (with no technical assumptions of smoothness of densities or integrability) characterizes Gamma and Kummer laws? Until now the only result from the literature is the equation
$$
\E(X^{\alpha}(1+X)^{-\beta}e^{\sigma X})\,\E(Y^{\beta} e^{\sigma Y})=\E(U^{\beta}(1+U)^{-\alpha}e^{\sigma U})\,\E(V^{\alpha}e^{\sigma V}),
$$
$\alpha,\beta\in\R$, $\sigma\le 0$, obtained in \cite{Kou14} (Lem. 5.1.).

Below we give the complete solution of the characterization problem. Again we will use the change of measure technique.
\begin{theorem}\label{tw3}
Let $X$ and $Y$ be independent positive non-degenerate random variables. Define $U$ and $V$ by \eqref{def1}. Suppose that $U$ and $V$ are independent. Then there exist $a,b,c>0$ such that
$X\sim \mathcal{K}\left(a,b-a,c\right)$, $Y\sim \mathcal{G}\left(b, c\right)$
\end{theorem}
\begin{proof}
Take $\eta<0$ and consider a random vector $(\X, \Y)$ with distribution $$ d\P _{\widetilde{X}, \widetilde{Y}}(x,y) = \frac{ e^{\eta(x+y)}d \P _{X,Y}(x,y)}{\E e^{\eta(X+Y)}}$$
which is defined on some probability space, possibly different than the space $X$ and $Y$ are defined on. 

We also set $\U = \frac{\Y}{1+\X}$ and
$\widetilde{V} = \widetilde{X}(1 + \widetilde{U})$. For $s, t\leq 0$, due to \eqref{u+v} we have

$$
\E\,e^{s \U + t \V} = \E\, e^{s \frac{\Y}{1+\X} + t \X(1+ \frac{\Y}{1+\X} )}= \tfrac{\E\,\exp\left[s \frac{Y}{1+X} + t X\left(1+\frac{Y}{1+X}\right)\right] e^{\eta(X+Y)}}{\E\,e^{\eta(X+Y)}}
=\tfrac{\E\, e^{s U + t V}\, e^{\eta(U+V)}}{\E\,e^{\eta(X+Y)}}
=\tfrac{\E \,e^{s U +\eta  U} \E\, e^{t V+\eta V}}{\E\,e^{\eta(X+Y)}}.
$$
Therefore
$$
\E e^{s \U + t \V} =\E\,e^{s \U}\,\E\,e^{t \V},
$$
where $\E\,e^{s \U}=\tfrac{\E\,e^{sU+\eta(X+Y)}}{\E\,e^{\eta(X+Y)}}$ and $\E\,e^{t \V}=\tfrac{\E\,e^{tV+\eta(X+Y)}}{\E\,e^{\eta(X+Y)}}$, and thus $\U$ and $\V$ are independent random variables.

Note that for $r>1$
\bel{war1}
\E (\V^r)=\tfrac{\E\, V^{r+1} e^{\eta(X+Y)}}{\E\,e^{\eta(X+Y)}}<\infty.
\ee
 Similarly, one can obtain the finiteness of $\E(\V^{r-1})$ and $\E(\V^{r+1})$.

 Due to independence of $\U$ and $\V$ we conclude that $\E(\V^{r-1}|\U)$, $\E(\V^r|\U)$ and $\E(\V^{r+1}|\U)$ are non-random. Therefore, the assumptions of Th. \ref{tw2} are satisfied with $\alpha_s = \E(\V^s )/\E(\V^{s-1})$ for $s\in\{r, r+1\}$. Thus, $\X\sim \mathcal{K}\left( \frac{\alpha_{r+1}}{\alpha_{r+1} - \alpha_r} -r , b-\frac{\alpha_{r+1}}{\alpha_{r+1} - \alpha_r} +r,\frac{1}{\alpha_{r+1} - \alpha_r} \right)$, ${\Y\sim \mathcal{G}\left(b,\frac{1}{\alpha_{r+1} - \alpha_r}\right)}$ where $b>0$. Eventually, we conclude that $X$ and $Y$ have required distributions with $a =\frac{\alpha_{r+1}}{\alpha_{r+1} - \alpha_r} -r>0$, $b>0$ and ${c= \frac{1}{\alpha_{r+1} - \alpha_r}-\eta>0}$ .

\end{proof}
\subsection{Remark on multivariate version}
In \cite{PW16} a multivariate analogue of the characterization of Kummer and gamma laws was considered. The approach was through a tree language and was in parallel to an earlier result of this kind involving a multivariate version of the Matsumoto-Yor property, see \cite{MW04} and \cite{B15}. However the characterization given in \cite{MW04} did not need any assumptions regarding density and its smoothness. Unfortunately, in the multivariate characterization given in \cite{PW16} (Th. 4) we assumed that random vector $\textbf{X}$ has density which is continuously differentiable. It was necessary for the first step in the inductive proof. Now, due to the result of Th. \ref{tw3}, we can omit this regularity condition. In order to state an improved version of Th. 4 from \cite{PW16}, first, we have to recall some definitions.

A graph $G=(V,E)$, where $V$ is the set of nodes and $E\subseteq \{\{u,v\}: u,v\in V, u\neq v\}$ is the set of edges, is called a \textit{tree}, if it is connected and acyclic. A vertex of degree 1 is called a \emph{leaf}.

Let $T=(V,E)$ be a tree of size $p\geq 2$. For a fixed root $r\in V$, we direct $T$ from the root towards leaves and denote such a directed tree by $T_r$. Having the tree directed, we can say that node $i$ is a \textit{child} of vertex $j$ (or $j$ is a \textit{parent} of $i$) if and only if $\{i,j\}\in E$ and the tree is directed from $j$ to $i$ (note that every node, unless it is a leaf, has at least one child and every node but the root has exactly one parent). The set of all children of $i$ in $T_r$ will be denoted by $\mathfrak{c}_r(i)$ and the parent of vertex $i$ by  $\mathfrak{p}_r(i)$.

For any $r\in V$ (and hence $T_r$) and fixed set of parameters $c_{i,j}>0$, $i,j\in V$ ($c_{i,i}=:c_i$) we define a transformation $\Phi^T_r: \mathbb{R}_+^p \rightarrow \mathbb{R}_+^p $ by
\bel{frt}\Phi_r^T (s_i, i\in V) = ( s_{i,(r)}, i\in V ),\ee where
\bel{gw}s_{i,(r)}=s_i\,\displaystyle{\prod_{j\in \mathfrak{c}_r(i)}\,\left(1+\frac{c_{i,j}}{c_i} s_{j,(r)}\right)},
\ee
where by convention an empty product is equal to $1$, i.e., if $V\ni i\ne r$ is a leaf then $s_{i,(r)}=s_i$. The definition  \eqref{gw} is inverse recursive with a starting point being any vertex with maximal distance from the root $r$.

\begin{corollary}\label{tw_char}
	Let $T=(V,E)$ be a tree of size $p$. Let $\mathbf{C}\in C_T$ and $\Phi_r^T$ be defined by \eqref{frt} and \eqref{gw}. Let $\mathbf{X}=(X_i, i\in V)$ be a $p$-dimensional random vector. Suppose that for every  $r\in V$, which is a leaf of $T$, the components of the random vector $\mathbf{X}_{(r)}=\Phi_r^T(\mathbf{X})= (X_{i,(r)},i\in V)$ are independent. Then there exist $\textbf{a}=(a_i, i\in V)\in (0,\infty )^p$ and $c>0$ such that $$X_{r,(r)}\sim \mathcal{G}(a_r, cc_r)\quad \textrm{and}\quad
	\frac{c_{\mathfrak{p}_r(i),i}}{c_{\mathfrak{p}_r(i)}}\,X_{i,(r)}\sim \mathcal{K}\left(a_i,a_{\mathfrak{p}_r(i)}-a_i,c\frac{c_{\mathfrak{p}_r(i)}\,c_i}{c_{\mathfrak{p}_r(i),i}}\right)\; \mbox{for } i\in V\backslash \{r\}.
	$$
\end{corollary}
\begin{proof}
	We use induction with respect to $p =|V|$. The case $p = 2$ can be obtained from Th. \ref{tw3} in the same way as Th. 3 in \cite{PW16}. Then we just follow the proof of Th. 4 of \cite{PW16}.
\end{proof}

\section{Characterizations related to the Koudou and Vallois property}
In this section we consider the property discovered in \cite{KV12} and discussed also e.g. in \cite{KV11} and \cite{Wes15}:  $X$ and $Y$ are independent and $U$ and $V$ are defined through \eqref{juvi}.

\subsection{Characterization through constancy of regressions}
 Our point of departure is the main result of \cite{Wes15} which is a counterpart of Th. \ref{tw1} in this paper. We recall it here.
\begin{theorem}\label{KV1}
Let $X$ and $Y$ be independent positive non-degenerate random variables and $\E\,X^{-1}<\infty$. Define $U$ and $V$ through \eqref{juvi}. If
\bel{reg}
\E(U|V)=\alpha \qquad \mbox{and}\qquad \E(U^{-1}|V)=\beta
\ee
for real constants $\alpha$ and $\beta$ then $0<\alpha<1<\beta$ and there exists a constant $c>0$ such that $$X\sim \mathcal{K}\left(1+\tfrac{1-\alpha}{\alpha\beta},\tfrac{(1-\alpha)(\beta-1)}{\alpha\beta},c\right)\qquad\mbox{and}\qquad Y\sim\mathcal{G}\left(\tfrac{(1-\alpha)(\beta-1)}{\alpha\beta},\,c\right).$$
\end{theorem}

Since $U\in(0,1)$ $\P$-a.s. one can consider conditional moments $\E(U^r|V)$, $r\ge 0$  without any additional assumptions. In the next result, we prove that some simple relations between such conditional moments characterize the Kummer and gamma laws of $X$ and $Y$. It is an analogue of Th. \ref{tw2} and the results o for the Lukacs and Matsumoto-Yor property obtained in \cite{CH03} and \cite{CH04}.

\begin{theorem}\label{KV2}
Let $X$ and $Y$ be independent positive non-degenerate random variables. Let $r\in\R$ be a fixed number. In case $r<1$ assume $\E\,X^{r-1}<\infty$. 

Define $U$ and $V$ through \eqref{juvi}. Let for $s\in\{r,r+1\}$ and $\alpha_s\in\R$
\bel{regr}
\E(U^s|V)=\alpha_s\,\E(U^{s-1}|V).
\ee

Then  $\alpha_r,\alpha_{r+1}\in(0,1)$, $\tfrac{\alpha_r(1-\alpha_{r+1})}{\alpha_{r+1}}>r-1$ and there exists $c_r>0$ such that $$X\sim\mathcal{K}\left(\tfrac{\alpha_r(1-\alpha_{r+1})}{\alpha_{r+1}}-r+1,\tfrac{(1-\alpha_r)(1-\alpha_{r+1})}{\alpha_{r+1}},c_r\right)
\qquad\mbox{and}\qquad Y\sim\mathcal{G}\left(\tfrac{(1-\alpha_r)(1-\alpha_{r+1})}{\alpha_{r+1}},\,c_r\right).$$
\end{theorem}

\begin{proof} Note that if $r\ge 1$ then all the moments we need: $\E\,U^{r+1}$, $\E\,U^r$ and $\E\,U^{r-1}$ are finite since $U\in(0,1)$ and in the case $r<1$ finiteness of these moments follows from the assumption $\E\,X^{r-1}<\infty$ since then $U^{r-1}<X^{r-1}$. 
	 	
Due to measurability of $\tfrac{X+Y}{1+X+Y}$ with respect to $\sigma(X+Y)$ we can rewrite \eqref{regr} as
$$
\E\left(\left.\left(\tfrac{X}{1+X}\right)^s\right|X+Y\right)\,\tfrac{1+X+Y}{X+Y}
=\alpha_s\E\left(\left.\left(\tfrac{X}{1+X}\right)^{s-1}\right|X+Y\right),\quad s=r,r+1.
$$
Equivalently, for $u\le 0$ we have
\bel{reg1}
\E\,\left(\tfrac{X}{1+X}\right)^s\tfrac{1+X+Y}{X+Y}\,e^{u(X+Y)}=\alpha_s\,\E\,\left(\tfrac{X}{1+X}\right)^{s-1}\,e^{u(X+Y)}.
\ee
Denote by $\tilde{X}$ a random variable with distribution
\bel{tix}
\P_{\tilde{X}}(dx)=\tfrac{\left(\tfrac{x}{1+x}\right)^r}{\E\,\left(\tfrac{X}{1+X}\right)^r}\,\P_X(dx),
\ee
which is defined on some probability space, possibly different than the probability space on which $X$ and $Y$ are defined.
Note that
\bel{inve}
\E\,\tilde{X}^{-1}=\tfrac{\E\,\tfrac{X^{r-1}}{(1+X)^r}}{\E\,\left(\tfrac{X}{1+X}\right)^r}<\infty.
\ee
Consider a rv $\tilde{Y}$, defined on the same probability space as $\tilde{X}$, in such a way that  $\tilde{X}$ and $\tilde{Y}$ are independent and $\tilde{Y}\stackrel{d}{=}Y$. Then, dividing both sides of \eqref{reg1} by  $\E\,\left(\tfrac{X}{1+X}\right)^r$, we obtain
\bel{reg3}
\E\,\tfrac{1+\tilde{X}+\tilde{Y}}{\tilde{X}+\tilde{Y}}\,e^{u(\tilde{X}+\tilde{Y})}
=\alpha_r\,\E\,\left(\tfrac{\tilde{X}}{1+\tilde{X}}\right)^{-1}\,e^{u(\tilde{X}+\tilde{Y})}
\ee
and
\bel{reg4}
\E\,\tfrac{\tilde{X}}{1+\tilde{X}}\tfrac{1+\tilde{X}+\tilde{Y}}{\tilde{X}+\tilde{Y}}\,e^{u(\tilde{X}+\tilde{Y})}
=\alpha_{r+1}\,\E\,\,e^{u(\tilde{X}+\tilde{Y})}.
\ee

Denote $\tilde{U}=\tfrac{1+\tfrac{1}{\tilde{X}+\tilde{Y}}}{1+\tfrac{1}{\tilde{X}}}$ and $\tilde{V}=\tilde{X}+\tilde{Y}$ and note that due to \eqref{inve} we have $\E\,\tilde{U}^{-1}<1+\E\,\tilde{X}^{-1}<\infty$. That is \eqref{reg3} and \eqref{reg4} are equivalent to
$$
\E\left(\tilde{U}^{-1}|\tilde{V}\right)=\tfrac{1}{\alpha_r}
\qquad\mbox{and}\qquad
\E\left(\tilde{U}|\tilde{V}\right)=\alpha_{r+1}.
$$

Then we see that \eqref{reg} with $\tilde{X}$ and $\tilde{Y}$ instead of $X$ and $Y$, respectively, and with $\beta=\tfrac{1}{\alpha_r}$ and $\alpha=\alpha_{r+1}$, is satisfied. Consequently, Th. \ref{KV1} implies  $\alpha_r,\,\alpha_{r+1}\in(0,1)$ and for some constant $c_r>0$
$$\tilde{X}\sim \mathcal{K}\left(1+\tfrac{\alpha_r(1-\alpha_{r+1})}{\alpha_{r+1}},\tfrac{(1-\alpha_r)(1-\alpha_{r+1})}{\alpha_{r+1}},c_r\right)\qquad\mbox{and}\qquad \tilde{Y}\sim\mathcal{G}\left(\tfrac{(1-\alpha_r)(1-\alpha_{r+1})}{\alpha_{r+1}},\,c_r\right).$$

Finally, from \eqref{tix} we conclude that the density of $X$ exists and is of the form
$$
f(x)\propto \tfrac{x^{1+\tfrac{\alpha_r(1-\alpha_{r+1})}{\alpha_{r+1}}-r}}{(1+x)^{\tfrac{1}{\alpha_{r+1}}-r}}\,e^{-c_rx}\,I_{(0,\infty)}(x).
$$
Consequently, integrability of $f$ implies
$\tfrac{\alpha_r(1-\alpha_{r+1})}{\alpha_{r+1}}>r-1$ and thus the distribution of $X$ is as asserted.
\end{proof}

Note that by taking a special positive value of $r$ in Th. \ref{KV2} the following result can be obtained immediately.

\begin{corollary}\label{r=2}
Let $X$ and $Y$ be independent positive non-degenerate random variables. For $U$ and $V$ defined in \eqref{juvi} assume that there exist constants $a$ and $b$ such that
\bel{rrr}
\E(U|V)=a\quad\mbox{and}\quad\E(U^2|V)=b.
\ee
Then $0<b<a<1$ and there exists a constant $c>0$ such that
$$
X\sim \mathcal{K}\left(\tfrac{a(a-b)}{b},\,\tfrac{(1-a)(b-a)}{b},\,c\right)\quad\mbox{and}\quad Y\sim\mathcal{G}\left(\tfrac{(1-a)(b-a)}{b},\,c\right).
$$
\end{corollary}
\begin{proof}
Note that \eqref{rrr} implies \eqref{regr} with $r=1$, $\alpha_1=a$ and $\alpha_2=\tfrac{b}{a}$.
\end{proof}

This result can be compared with one of the main cases considered in \cite{BH65} (see (i) od its Sec. 3), which is its analogue for the Lukacs property.

For a special negative value $r$ in Th. \ref{KV2} we obtain another interesting particular case.

\begin{corollary}\label{dwa}
Let $X$ and $Y$ be independent positive non-degenerate random variables and $\E\,X^{-2}<\infty$. For $U$ and $V$ defined in \eqref{juvi} assume that there exist constants $a$ and $b$ such that
\bel{rrr1}
\E(U^{-1}|V)=a\quad\mbox{and}\quad\E(U^{-2}|V)=b.
\ee
Then $1<a<b$ and there exists a constant $c>0$ such that
$$
X\sim \mathcal{K}\left(\tfrac{a(a-1)}{b},\,\tfrac{(a-1)(b-a)}{b},\,c\right)\quad\mbox{and}\quad Y\sim\mathcal{G}\left(\tfrac{(a-1)(b-a)}{b},\,c\right).
$$
\end{corollary}
\begin{proof}
Note that \eqref{rrr1} implies \eqref{regr} with $r=-1$, $\alpha_0=\tfrac{1}{a}$ and $\alpha_{-1}=\tfrac{a}{b}$.
\end{proof}
Similarly, as Cor. 2.3 of the previous section, this result can be compared with Th. 1.1 in \cite{LHH94} which is its analogue for the Lukacs property.

\subsection{Characterization through independence}
We conclude the paper with the characterization of the Kummer and gamma laws by independence of $U$ and $V$ without any smoothness or moments conditions. We utilize the results obtained above for regressions.
\begin{corollary}
	Let $X$ and $Y$ be independent positive non-degenerate random variables. If $U$ and $V$ defined in \eqref{juvi} are independent then $X$ has a Kummer distribution and $Y$ has a gamma distribution.
\end{corollary}
\begin{proof}
	Since $U$ is a $[0,1]$-valued it has all moments finite. Consequently, independence of $U$ and $V$ implies that conditions \eqref{rrr} are satisfied and the result follows by Cor. \ref{r=2}.
\end{proof}

\vspace{5mm}\noindent\small {\bf Acknowledgement.}  We are grateful to G. Letac for sending us his unpublished paper on the Kummer distribution.

This research has been supported by the grant 2016/21/B/ST1/00005 of National Science Center, Poland.

\end{document}